\documentclass[11pt]{article}
\usepackage[dvips]{graphicx}
\usepackage{latexsym}
\usepackage{amsmath,amssymb,amsfonts,amsthm}
\usepackage[utf8]{inputenc}
\usepackage{xcolor}
\usepackage{enumerate}

\newcommand \NN {\mathbb{N}}

\newcommand \ZZ {\mathbb{Z}}

\newcommand \PP {\mathbb{P}}

\newcommand \DP {\overset{def}{=}}

\newtheorem{definition}{Definition}[section]
\newtheorem{theorem}{Theorem}[section]

\newtheorem{lemma}{Lemma}[section]
\newtheorem{remark}{Remark}[section]

\author{%
 Serge Cohen \footnote{ Institut de Math\'ematiques de Toulouse; UMR 5219, Université de Toulouse; CNRS, UT3 F-31062 Toulouse Cedex 9, France. First-Name.Name@math.univ-toulouse.fr}
  \and %
 Xavier Bressaud \footnotemark[1] 
  }

\title{Transition of the simple random walk on the graph of the ice-model} 

\begin{document}
\maketitle

\begin{abstract}
The $6$-vertex model is a seminal model for many domains in Mathematics and Physics. The sets of configurations of the $6$-vertex model can be described as the sets of paths in multigraphs.  In this article the transition probability of the simple random walk on the multigraphs is computed. The unexpected point of the results is the use of continuous fractions to compute the transition probability. 
\end{abstract}

\noindent \textit{Keywords}: Random walk, Markov Chain
\noindent \textit{AMS classification (2000)}: 05C81, 60F05.
\section{Introduction}
 In this article we are interested in  a simple random walk $ Y $  on particular 
 (multi)graphs $ G_K $ indexed by an integer $ K.$ The set of vertices 
 of $ G_K $ is $V_K \DP \lbrace -1,1\rbrace^K $ and the set $ E_K $ of edges 
 is defined so that the set of paths of length $ n $ is isomorphic with 
 the set of configurations of the so-called $6$-vertex model, on a rectangle $K \times n $ in the particular  case of the ice-model~\cite{Lieb67}. Please note that the uniform  distribution is usually considered in the $6$-vertex model, with various admissible boundary conditions on the rectangle when $n, \ K \to \infty$ of a rectangular lattice. In this paper 
we endow the sets of  configurations with the distribution of $ Y_0,\ldots, \; Y_n $ 
which is actually easier to study that the uniform distribution on  the set of paths of length $ n.$ In the $6$-vertex model the height function  $ h $ which is a map 
from the rectangles $K \times n $ to the set of integers $\ZZ$ has physical meaning, and more precisely the decay of the variance of the difference of height  function between two distant points  on the rectangle is of special importance. This question is also related to random graph homomorphisms~\cite{Benjamini07}. See~\cite{Duminil21}
 where the  decay is shown to be of logarithmic order for periodic boundary conditions 
 on the rectangle. For the Markov chain $ Y $ the height fonction is related to an  additive functional $ Z $ of the $ Y,$ which collects a particular vector field defined 
on the edges of $ G_K. $ In a previous paper \cite{Boissard15} it is shown that the variance of the height 
function decays like $ \frac2{(2+K)n}.$ Still for finite $ K $ the variance is computed 
in~\cite{Espinasse16} for periodic boundary condition in the variable between $1$ and $ K,$ and in~\cite{Lammers22} for other constraints.

The easiest way to describe the simple random walk on $G_K $ in the stationary regime is to state that the distribution of the pair $(Y_K(0),Y_K(1)) $ is the uniform  
distribution on the edges $ E_K.$  One aim of this article is to provide a formula 
for the transition probability of the simple random walk that starts from a given vertex. 
 Surprisingly enough the formula in the Theorem~\ref{the:transi} uses the continuous fraction associated to the length of the constancy blocks of the digits of the given 
vertex. The authors see two useful consequences of this result. The first one is the fact that the transition probability of the simple random walk will considerably 
make easier and faster simulations of the walk. See~\cite{Montegut2020} for a survey of previous simulation methods, which are both memories and computationally intensive due the difficulty to describe easily the neighbors of a given vertex in $ G_K $
for large $K$. Another interesting consequence of Theorem~\ref{the:transi} is that it can be extended to the case $ K= \infty.$ This in turns gives a sense to the decay 
of the variance of the height function in $ G_\infty.$ The definitions and the models are given in section~\ref{sec:model}. In section~\ref{sec:results} the main theorems are written and they 
are proved in the following sections.


\section{The model}
\label{sec:model}
Let $\left(Z^{(1)}_t, \ldots, Z^{(K+1)}_t \right)_{ t \in \mathbb N}  \in \mathbb{Z}^{K+1}$ denote the heights of $K+1$ 
simple random walks on $\mathbb{Z}$, conditioned on satisfying

\begin{equation}\label{ContrForme}
\forall t \in \mathbb{N}, \forall i \in \left[ 1, K\right], \left| Z^{(i+1)}_t - Z^{(i)}_t \right| = 1. 
\end{equation}

More precisely, the random walk is a Markov chain on the state space of $K$-step walks in $\mathbb Z$

\begin{equation}
\label{ContrMarchZ}
 \mathcal{S}_K = \{ (z^{(1)}, \ldots, z^{(K+1)}) \in \mathbb{Z}^{K+1}, \quad \forall i \in \left[ 1, K\right], |z^{(i+1)} - z^{(i)}| = 1 \}
\end{equation}

\noindent where the next step from $z_0 \in \mathcal{S}_K$ is selected uniformly among  the $z_1$s that belong to $\mathcal{S}_K$ 
such that  $ \forall i \in [1,K+1], \;  z_{1}^{(i)}- z_{0}^{(i)} \in \{-1,1 \}.$
 In other words, we consider $K+1$ simple
random walks on the lattice $\mathbb{Z}$ coupled under a shape condition.
One can associate to a path of length $n$ $ (Z^{(i)}_t)_{0\le t \le n-1} $ a height function $h$ on a rectangle $ ((t,i))_{0 \le t \le n-1, \; 1 \le i \le K+1} $ by 
$ h(t,i)=Z^{(i)}_t, $ which makes the link with the height function of the $6$-vertex model. 

In~\cite{Boissard15} an equivalent depiction is provided as a simple random walk on a 
(multi)-graph. 
\begin{definition}\label{DefinGraphGK}
Let $V_K\DP \lbrace -1,1\rbrace^K$,  $E_K^+,E_K^-\in V_K\times V_K$ 
will be respectively set of postive, negative edges on $V_K.$

A pair $ (a,b)\in V_K\times V_K$ such that  $a\neq b$ belongs to  $E_K^+$ if non vanishing coordinates of the vector  $(b-a)\in \{ -2,0,2\}^K$ have alternate signs with the first sign \emph{negative}. For every vertex $a\in V_K$ there is an edge from  $a$ to $a$  $E_K^+$ denoted by  $(a,a)^+$. 

In a similar manner there is a pair $ (a,b)\in V_K\times V_K$ such that  $a\neq b$ belongs to  $E_K^-$ if non vanishing coordinates of the vector  $(b-a)\in \{ -2,0,2\}^K$ have alternate signs with the first sign \emph{positive}. For every vertex $a\in V_K$ theres is an edge from  $a$ to $a$  $E_K^-$ denoted by  $(a,a)^-$.

For every  $K\in\NN$ the (multi)graph $G_K\DP (V_K,E_K)$, where  $E_K\DP E_K^+\cup E_K^-$. 
\end{definition}

The study of  $(Z_t^{(1)},...,Z_t^{(K+1)})$ can be split in the study of the first coordinate $Z_t^{(1)}$ and of the increments  $$Y_K(t)\DP \Big(Z_t^{(2)}-Z_t^{(1)},\dots ,Z_t^{(K+1)}-Z_t^{(K)}\Big),$$ which is always an   element of $V_K$ (because of~\eqref{ContrForme}).

Let us remark that the uniform distribution on the set of edges $E_K$ is the same 
as the distribution of $ (Y_K(0),Y_K(1)) $ if the Markov process $Y_K$ is stationary. 
We have another useful characterization of this distribution given by the following Lemma. 

\begin{lemma}\label{DescrSeque}
Let  $\epsilon$ be  a  random variable such that $\PP(\epsilon=-1)=\PP(\epsilon=1)=\frac{1}{2}$), let $(\alpha_k)_{k\geq 1}$ be an  i.i.d. sequence of  Bernoulli random variables with parameter $\frac{1}{3}$ and let $(\beta_k)_{k\geq 1}$ be an i.i.d. sequence of random variables such that $\PP(\beta_k=-1)=\PP(\beta_k=1)=\frac{1}{2}$. The previous random variables are mutually independent. Let us define  $\forall k\geq 2$ :  
\[\gamma_k \DP \epsilon(-1)^{\sum_{i=1}^{k-1}\alpha_i},\] 
with the convention that $\gamma_1=\epsilon$.

Let us also  define  $k\geq 1$ :
\begin{eqnarray}
\label{eq:seq-G_K}
A_k &=& (1-\alpha_k)\beta_k+\alpha_k\gamma_k \\
B_k &=& (1-\alpha_k)\beta_k-\alpha_k\gamma_k. 
\end{eqnarray}

The distribution of the pair $((A_k)_{1\leq k\leq K},(B_k)_{1\leq k\leq K})$ is the  uniform distribution on  $E_K$. If we denote by  $(A,B)_K$ the edge defined by  :
\[(A,B)_K \DP \left\lbrace \begin{array}{ll}
    \left((A_k)_{1\leq k\leq K},(B_k)_{1\leq k\leq K}\right) & \text{ if } (A_k)_{1\leq k\leq K}\neq (B_k)_{1\leq k\leq K} \\
    \left((A_k)_{1\leq k\leq K},(B_k)_{1\leq k\leq K}\right)^\epsilon & \text{ if } (A_k)_{1\leq k\leq K}= (B_k)_{1\leq k\leq K}
\end{array}\right.\]
then $(A,B)_K\overset{\mathcal{L}}{=}(Y_K(0), Y_K(1))$ if $ Y_K(0) \neq Y_K(1) $ and 
$$ (A,A)^{\epsilon}_K\overset{\mathcal{L}}{=}(Y_K(0), Y_K(0))^{Z_1^{(1)}-Z_0^{(1)}} \quad  \text{if} \quad Y_K(0) = Y_K(1).$$   Moreover $(A_k)_{1\leq k\leq K}\overset{\mathcal{L}}{=}Y_K(0)$ and $(B_k)_{1\leq k\leq K}\overset{\mathcal{L}}{=}Y_K(0)$.
\end{lemma}
\begin{proof}
The proof is by induction and can be found in ~\cite{Montegut2020}. 
Let us first prove that $(A,B)_K \in E_K.$ We remark that $ B_k - A_k = - 2 \alpha_k \gamma_k,$ then it is vanishing if $ \alpha_k = 0,$ and the alternating rule sign is fulfilled because every time $\alpha_k = 1,$ $ \gamma_k $ has a different sign from 
$\gamma_{k-1}.$ Let us denote by $D_K$ the cardinal of $ E_K.$ A simple 
computation yields $D_1=6$ and by induction $D_K = 2 \times 3^K.$ 
It is also obvious to check $ \PP ((A,B)_1 = e) =  \frac16 $ for every edge in $E_1.$
Let us assume that  $ \PP ((A,B)_K= e) =  \frac1{D_K}$ is true for every edge in $E_K$
Let us consider $u', \; v' \in V_K$ and denote by $u'\pm 1$ the vertex in $V_{K+1}$ 
obtained by concatenating $\pm 1 $ on the right of $ u'.$ Then for 
\begin{align*}
 \PP ((A,B)_{K+1}= (u'1,v'1)) &=  \PP ((A,B)_K= (u',v') \cap \alpha_{K+1} = 0 \cap \beta_{K+1}= 1) \\
 &= \frac1{D_K} \frac23 \frac12 \\ 
 & = \frac1{D_{K+1}}. 
\end{align*}
 The same 
holds for $(u'-1,v'-1).$ If the concatenated digit to $u'$ is different from the one concatenated to $v'$ and $ u' \neq v'$ there is only  one possible choice which leads 
to a non vanishing probability depending on the last digit that differs between 
$u'$ and $v'.$ 
Then for this choice 
\begin{align*}
  \PP ((A,B)_{K+1}= (u' \pm 1,v' \mp 1)) &=  \PP ((A,B)_K= (u',v') \cap \alpha_{K+1} = 1 ) \\ &= \frac1{D_K} \frac13 \\
  &= \frac1{D_{K+1}}.
\end{align*} 
The proof is complete when we consider the case $u'=v'$ and in this case we get also 
$ \PP ((A,B)_{K+1}= (u' \pm 1,v' \mp 1))  =  \frac1{D_{K+1}} $ thanks to the distribution of $\epsilon.$ 
\end{proof}
\begin{remark} 
\label{rem=infty}
One important consequence of the previous result is the fact that the graph 
$G_K $ is defined for $ K= \infty.$ Let us be more precise. 
When $ K = \infty, $ $ V_\infty \DP \lbrace -1,1\rbrace^{\NN}, $ and 
$ E_\infty^+, \; E_\infty^- $ are defined with the same alternating rules as for finite 
$K.$  The Lemma~\ref{DescrSeque} is still true when $K$ infinite. 
\end{remark}
\section{Results}
\label{sec:results}
 The aim of this article is to compute the transition probability of the stationary Markov chain associated with the simple random walks on the (multi)graphs
 $ G_K $ for  $ K  \in \NN \cup \infty.$ Hence we will compute the conditional probability that $ (B_k)_{k\le K} $ takes a particular value in $ V_K$ once the sequence $ (A_k)_{k\le K}$ is given. 
 
  Please note that when $ K = \infty, $ the law of large number implies for both sequences $ (A_k)_{k \in \NN }, \; (B_k)_{k \in \NN }$  that there are not constant for  $ k $ big enough almost surely.  
  
Let us assume that the sequence 
$ (a_k)_{k \in \NN } \in \{-1,1 \}^{\NN}  $ starts with $ a_1 = 1.$ 
Let us fix the consecutive times where $ a $ is constant and denote 
by $ i_m = (-1)^{m+1}.$ (If $a_1=-1,$ then $ i_m = (-1)^m$).
By convention we set $ S_0=0,$ and for $ m \ge 1,$  we assume that the $m$-th block of constancy of $ a$ starts with $ S_{m-1}+1 $ and stops with $ S_m.$

Let us denote for $ m \ge 1,$  the event 
\begin{equation}
A^{(m)}= \{ a \in \{-1,1\}^{\NN} \; \mbox{such that} \;  a_{S_{m-1}+1} = i_m, \ldots, \; a_{S_{m}}= i_m \} 
\end{equation}
 of sequences which are equal to $ a$ on the  $m$-th block of constancy of $ a$. 
 
\begin{remark} 
In the following we are conditioning the distribution of $ B $  with respect 
of events of the form  $ \{ A = a\},$ where $ a $ is a deterministic sequence. Once $ a $ is given, so is the sequence $ S $ and the conditioning with respect of $ A^{(m)} $ 
actually means with respect of the event $ \{ A_{S_{m-1}+1} = i_m, \ldots, \; A_{S_{n}}= i_m  \}.$ We will use the abuse of notation $ \PP (.|a), \; \PP (.|A^{(m)}) $ in the sequel. 
\end{remark}

  Hence $ \{ a \} = \{(a_k)_{k \le K} \} = \cap_{m=1}^N A^{(m)}  $ where $ N $ is the number of blocks of constancy of $a.$ By definition of $M_m$ the length of the $m$-th block of constancy of $ a $ is 
 equal to 
 \begin{equation}
 M_m = S_m - S_{m-1}. 
 \end{equation}
 
Let   for $ 1 \le m \le n \le N $ 
\begin{equation}
\label{eq:frac-cont}
x^n_m \DP \frac{1}{M_m + \frac{1}{M_{m+1}+\frac{1}{\ldots+ \frac1{M_n+1}}}} = [M_m, M_{m+1},\ldots, M_n,1],
\end{equation}
if $ n < m $ we set $ x^n_m \DP 1$ by convention.

When $ K = \infty$ the continuous fraction in~\eqref{eq:frac-cont} is converging when $ 
n \to  \infty$  toward an irrational number because of Remark~\ref{rem=infty} that 
will be denoted by $x_m^{\infty}.$ 

Please remark that if $Y_K(0) = a $ and $Y_K(1) = b,$ at most one digit $b_k$ of $ b $ is different of $a_k $ in any block of constancy $ A^{(m)} $ of $ a,  $ because of the alternating sign rule. Let $(\epsilon_m)_{1 \leq m \leq N}$ be   Bernoulli random variables such that $\epsilon_m =1 $ if and only if there is one change of digits between $ a$ and $ b $ in the $m$-th  block of constancy $ A^{(m)}.$ Let  $(E_m)_{1 \leq k\leq N}$  be a sequence of independent random variables uniformly distributed on $\{1, \ldots, M_m\}$ which encode the digit that is changed in $ A^{(m)}.$ One further constraint due to the alternating sign rule is that when $\epsilon_m =1, $ $\epsilon_{m+ 2 k} = 0 $ on the event that $ \epsilon_{m+1}= 0, \ldots, \epsilon_{m+(2k-1)}=0.$  In other words there cannot be change of digits in two consecutive blocks of constancy of $ a $ 
that have an even difference of indexes, since the $a_k $ are the same on those blocks. 
The conditional probability $ \PP(Y_K(1) = b | Y_K(0) = a)$  is then described by the following Theorem that yields the distribution of the $(\epsilon_m)_{1 \leq m \leq N}.$ 
\begin{theorem}
\label{the:transi}
The distribution of $ (\epsilon_m)_{1 \leq m \leq N}$  is given by~:
\begin{itemize}
\item Initializing phase 
\begin{equation}
\label{eq:eps1}
\PP ( \epsilon_1= 1| a  )=\frac{M_1}{M_1+1 +x_2^N} 
\end{equation} 
where $N$ is the number of blocks of constancy of $a.$ 
\item Subsequent phase when previously there is no change
\begin{equation}
\label{eq:sub-1}
 \PP(\epsilon_m= 1|\epsilon_{m-1}= 0, \ldots, \epsilon_1=0,\; a) = 
\frac{M_m}{M_m + 1 + x_{m+1}^{N}}
\end{equation}
where $ N  $ is the number of blocks of constancy of $a.$ 
\item Subsequent phase when previously there is at least one  change $$ \PP(\epsilon_m= 1|\epsilon_{m-1}= 0, \ldots, \epsilon_{m-(2k-1)}=1,\; a)= M_m x_m^N,$$ where $ N  $ is the number of blocks of constancy of $a.$ 
\item Loop in $ G_K.$ If $\epsilon_m = 0 $ for $ 1 \leq m \leq N $ which is  equivalent to $  b= a.$ 
\begin{align}
\PP((Y_K(0),Y_K(1)) &= (a,a)^{+} |  Y_K(0) = a) \notag \\ &= \PP((Y_K(0),Y_K(1)) = (a,a)^{-} |  Y_K(0) = a) \notag \\ 
&= \frac12 \PP (\epsilon_m = 0, \text{ for }  1 \leq m \leq N |a). \label{eq:loop}
\end{align}
\end{itemize}
$ \forall m \ge 1 $ the distribution of the $ E_m$'s is uniform on 
$ \{ 1,\ldots,M_m\} $  conditionally to the event $\epsilon_m = 1.$ 
\end{theorem}
\begin{remark}
Remember that the distribution $ \PP (.|a) $ is the uniform distribution on 
the neighbors of $ a $ in $E_K,$ this fact is not obvious from the previous theorem. 
Indeed, if   a given $ a $ in $E_K $ has $ N $ blocks of constancy,  it leads for instance to 
\begin{align}
\label{eq:equi}
\frac1{deg_K(a)} &= \frac{\PP( (\epsilon_1,\ldots,\epsilon_N)=(1,\dots,1)| a)}{\prod_{k=1}^N M_k} \\
&=  \frac1{M_1+1+ x_2^N}\prod_{k=2}^{N} x_k^N. \notag
\end{align}
Let us now suppose  that for a given $ 2 \le k_0  \le k_0+1 < N,$ $ \epsilon_{k_0}=0,$ it implies that $ \epsilon_{k_0+1}=0,$  and the other $ \epsilon_k =1.$ The equation 
\begin{equation}
\label{eq:equi-2}
\frac1{deg_K(a)}= \frac{\PP( (\epsilon_1,\ldots,\epsilon_N)=(1,\dots,1,0,0,1,\dots,1)| a)}{\prod_{k=1}^N M_k}
\end{equation}
is still true when the $0,0$ are for this $k_0.$ In equation~(\ref{eq:equi}) we only have to change the factors for $k_0$ and $ k_{0}+1, $ so  $ x_{k_0}^N $  becomes $ x_{k_0+1}^N x_{k_0}^N $ and the factor  $ x_{k_0+1}^N $ becomes $1.$ Then we check
$$ \frac{\PP( (\epsilon_1,\ldots,\epsilon_N)=(1,\dots,0,0,\dots,1)| a)}{\prod_{k=1}^N M_k} = \frac1{deg_K(a)}.$$
  Tedious computations can show that actually the probability to jump from $ a $ to each of his neighbor is the same.
\end{remark}

When $ K = \infty,$ the previous Theorem still holds true 
when we consider that  the number of constancy blocks of $a$ is infinite and use 
the definition of the continuous fraction as a limit . 
\begin{theorem}
\label{the:transi-infty}
The distribution of $ (\epsilon_m)_{1 \leq m}$  is given by~:
\begin{itemize}
\item Initializing phase 
\begin{equation}
\label{eq:eps1-infty}
\PP ( \epsilon_1= 1| a  )=\frac{M_1}{M_1+1 +x_2^\infty}.
\end{equation} 
\item Subsequent phase when previously there is no change
\begin{equation}
\label{eq:sub-1-infty}
 \PP(\epsilon_m= 1|\epsilon_{m-1}= 0, \ldots, \epsilon_1=0,\; a) = 
\frac{M_m}{M_m + 1 + x_{m+1}^{\infty}}.
\end{equation}
\item Subsequent phase when previously there is at least one  change $$ \PP(\epsilon_m= 1|\epsilon_{m-1}= 0, \ldots, \epsilon_{m-(2k-1)}=1,\; a)= M_m x_m^\infty.$$ 
\end{itemize}
$ \forall m \ge 1 $ the distribution of the $ E_m$'s is uniform on 
$ \{ 1,\ldots,M_m\} $  conditionally to the event $\epsilon_m = 1.$ 
\end{theorem}
\begin{remark}
Please note that the probability of a loop in $G_\infty $ is vanishing. Hence 
there is no loop case in the last Theorem. 
\end{remark}

\section{Proof of the result}
 Let us now introduce the conditional independence with respect of $ \gamma_k,$ which is 
an important tool for our computations.
Let us denote by $ \sigma_l^m \DP \sigma( a_k,\; b_k, \; l \le k \le m ),$ 
for $2 \le l \le m.$ By convention $ \sigma_1^m \DP  \sigma( a_k,\; b_k, \; 1 \le k \le m, \epsilon). $ 
\begin{lemma}
\label{lem:cond-ind}
For every $  1 \le l  \le m \le n  $ $  \sigma_l^m $ is  independent of $ \sigma_{m+1}^n$ conditionally to $ \gamma_{m+1}.$ 
\end{lemma}
The proof of this Lemma comes from the definitions of Lemma~\ref{DescrSeque}.


For every $  m\le n,$ let us define $$ u_m^n \DP \PP ( \gamma_{S_{m-1}+1}= i_m \cap A^{(m)} \cap  \ldots \cap A^{(n)}) $$ and  $$ v_m^n \DP \PP ( \gamma_{S_{m-1}+1}= -i_k \cap A^{(m)} \cap  \ldots \cap A^{(n)}).$$ 
With the help of Lemma~\ref{lem:cond-ind}
we can compute $ u_m^n, v_m^n, $ by induction starting from  $ u_n^n, v_n^n, $
and we get the following result. 
\begin{lemma}
\label{lem:rec-matrix}
For every $   m <  n,$
$$ \begin{pmatrix} u_m^n \\ v_m^n \end{pmatrix}  =  \frac1{3^{M_m}} \begin{pmatrix}

 M_m & 1 \\ 1 & 0 \end{pmatrix}  \begin{pmatrix}u_{m+1}^n \\ v_{m+1}^n \end{pmatrix} $$
\end{lemma}
\begin{proof}
By conditioning and Lemma~\ref{lem:cond-ind} we may write
\begin{multline*}
 u_m^n  =\PP (\gamma_{S_{m-1}+1}= i_m, \; A^{(m)} | \gamma_{S_{m}+1}= i_m )  \\ \times \PP (\gamma_{S_{m}+1}= - i_{m+1} , \; A^{(m+1)} \cap  \ldots \cap A^{(n)} ) \\
  +  \PP (\gamma_{S_{m-1}+1}= i_m , \; A^{(m)} | \gamma_{S_{m}+1}= - i_m ) \\ \times  \PP (\gamma_{S_{m}+1}=  i_{m+1} , \; A^{(m+1)} \cap  \ldots \cap A^{(n)} ).
\end{multline*}
This equation yields 
\begin{align}
\label{eq:u_m^n}
 u_m^n  &= \PP (\gamma_{S_{m-1}+1}= i_m \cap A^{(m)} | \gamma_{S_{m}+1}= i_m )  v_{m+1}^n \\ &+  \PP (\gamma_{S_{m-1}+1}= i_m \cap A^{(m)} | \gamma_{S_{m}+1}= - i_m )   u_{m+1}^n.  \notag
\end{align}
Observe that 
\begin{multline}
  \{ \gamma_{S_{m-1}+1}= i_m \cap A^{(m)} \cap \gamma_{S_{m}+1}= i_m \} = \\
\{ \alpha_l = 0,\; \beta_l= i_m \; \mbox{for} \; S_{m-1}+1 \le  l \le S_m \} \cap  \{ \gamma_{S_{m-1}+1}= i_m \}.
\end{multline}
The two events on the right hand side are independent and the probability of the first one 
is $ \frac1{3^{M_m}}.$ Therefore
\begin{multline*}
\label{eq:prob-cond1}
\PP (\gamma_{S_{m-1}+1}= i_m \cap A^{(m)} | \gamma_{S_{m}+1}= i_m )   = \\
\frac{ \PP ( \gamma_{S_{m-1}+1}= i_m \cap A^{(m)} \cap \gamma_{S_{m}+1}= i_m)}
{\PP (\gamma_{S_{m}+1}= i_m  ) }.
\end{multline*}
Hence
\begin{equation}
\label{eq:prob-cond1}
\PP (\gamma_{S_{m-1}+1}= i_m \cap A^{(m)} | \gamma_{S_{m}+1}= i_m )   = \frac1{3^{M_m}}.
\end{equation}

With a similar argument we get 
\begin{equation}
\label{eq:prob-cond2}
 \PP (\gamma_{S_{m-1}+1}= i_m \cap A^{(m)} \cap \gamma_{S_{m}+1}= -i_m) = \frac{M_m}{3^{M_m}}.
\end{equation}
Moreover 
\begin{multline*}
 v_m^n  =\PP (\gamma_{S_{m-1}+1}=- i_m \cap A^{(m)} | \gamma_{S_{m}+1}= - i_m ) \\ \times \PP (\gamma_{S_{m}+1}=  i_{m+1} \cap A^{(m+1)} \cap  \ldots \cap A^{(n)} ) \\
  +  \PP (\gamma_{S_{m-1}+1}= - i_m \cap A^{(m)} | \gamma_{S_{m}+1}= i_m ) \\ \times \PP (\gamma_{S_{m}+1}= - i_{m+1} \cap A^{(m+1)} \cap  \ldots \cap A^{(n)} ).
\end{multline*}
Note that 
$ \PP (\gamma_{S_{m-1}+1}= - i_m \cap A^{(m)} \cap \gamma_{S_{m}+1}= i_m) =0,$  
since, on this event, the value $- i_m $ of $ \gamma_{S_{m-1}+1}$ does not fit the value of $a_l$'s 
on  $A^{(m)}.$ Hence all the $\alpha_l= 0 $ and $ \gamma_{S_{m}+1} = \gamma_{S_{m-1}+1}.$ It follows that  
\begin{equation}
\label{eq:v_m^n}
 v_m^n = \PP (\gamma_{S_{m-1}+1}=- i_m \cap A^{(m)} | \gamma_{S_{m}+1}= - i_m ) u_{m+1}^n. 
\end{equation}
The Lemma~\ref{lem:rec-matrix} is the consequence of equations (\ref{eq:u_m^n}), (\ref{eq:v_m^n}), (\ref{eq:prob-cond1}), (\ref{eq:prob-cond2}). 
\end{proof}
 
Let us recall~(\ref{eq:frac-cont}) $ x_m^n \DP\frac{v_m^n }{u_m^n }. $ Lemma~\ref{lem:rec-matrix} 
yields for $ m+1 \le n,$ 
\begin{equation}
\label{eq:x_m^n_induc}
 x_m^n= \frac1{M_m + x^n_{m+1}}. 
\end{equation}
Moreover 
\begin{align}
u_n^n &= \PP ( \gamma_{S_{n-1}+1}= i_n, \; a_j= i_n, \; S_{n-1}+1 \le j \le S_n) \\
&= \PP ( \gamma_{S_{n-1}+1}= i_n,  \beta_j = i_n,\; \alpha_j = 0,\;  S_{n-1}+1 \le j \le S_n) \notag \\
&+ \sum_{j= S_{n-1}+1}^{S_n} \PP ( \gamma_{S_{n-1}+1}= i_n, \; \forall j \neq j_0 \beta_j= i_n,\; \alpha_j=0, \; \mbox{and} \; \alpha_{j_0}=1) \notag  \\
&= \frac12 \frac{1}{3^{M_n}}+ \frac12 \frac{M_n}{3^{M_n}}. \label{u_n^n}
\end{align}
Moreover $ v_n^n= \frac12 \frac{1}{3^{M_n}}$ and $x_n^n=\frac1{M_n+1}.$
Hence for $ k\le n $ 
\begin{equation}
\label{eq:frac-cont-x^k_n}
x^n_m = \frac{1}{M_m + \frac{1}{M_{m+1}+\frac{1}{\ldots+ \frac1{M_n+1}}}} = [M_m, M_{m+1},\ldots, M_n,1].
\end{equation}
\begin{remark}
Please note that we can use the induction of Lemma~\ref{lem:rec-matrix} even if $ K=\infty.$ Moreover when $ n\to \infty $ $ x^n_m $ is converging to the irrational
number that we will denote by $ x^\infty_m $
\end{remark}
\section{Initializing phase}
In this part we will compute the  conditional probability given $a$ that there is  a digit $ b_j \neq a_j $ in the first block of constancy of $a$ namely $ A^{(1)}.$  
Hence when $ K $ is finite if $ a $ has $N$ blocks of constancy we are aiming 
for $ \PP (\epsilon_1= 1 | \cap_{m=1}^N  A^{(m)} ).$ When $ K$ is infinite we want to compute the limit of the previous probability when $ N \to \infty.$ 
For $ N \ge n \ge 2,$ we start by computing 
\begin{align*}
\PP (& \epsilon_1= 0 , \; A^{(1)}, \; \gamma_{S_{1}+1}=\pm1, \; \cap_{m=2}^n  A^{(m)}) \\
&=  \PP ( \epsilon_1= 0, \; A^{(1)}, \; \cap_{m=2}^n  A^{(m)} |  \; \gamma_{S_{1}+1}=\pm1)\PP ( \gamma_{S_{1}+1}=\pm1) \\
&= \PP ( \epsilon_1= 0, \; A^{(1)} | \gamma_{S_{1}+1}=\pm1) \\
&\phantom{\PP ( \epsilon_1= 0,)}\times  \PP ( \cap_{m=2}^n  A^{(m)},   \; \gamma_{S_{1}+1}=\pm 1| \gamma_{S_{1}+1}=\pm1) \PP ( \gamma_{S_{1}+1}=\pm1) \\
&=\frac{ \PP ( \epsilon_1= 0, \; A^{(1)},\;  \gamma_{S_{1}+1}=\pm1) \PP ( \cap_{m=2}^n  A^{(m)} ,  \; \gamma_{S_{1}+1}=\pm1) }{\PP ( \gamma_{S_{1}+1}=\pm1)}.
\end{align*} 
Since $ \PP ( \gamma_{S_{1}+1}= \pm 1) =\frac12$ and  
$$ \PP ( \epsilon_1= 0,  \; A^{(1)},\;  \gamma_{S_{1}+1}=\pm1) = \PP ( \epsilon_1= 0, \; A^{(1)},\;  \epsilon=\pm1));$$
\begin{multline*}
 \PP ( \epsilon_1= 0,  \; A^{(1)},\;  \gamma_{S_{1}+1}=\pm1) 
=  \\ \PP (\forall k = 1 \; \mbox{to} \; M_1,  \; \alpha_k=0 \;  S_1 \beta_k = a_1, \; \epsilon=\pm1),
\end{multline*}
$$ \PP ( \epsilon_1= 0,  \; A^{(1)},\;  \gamma_{S_{1}+1}=\pm1) = \frac1{3^{M_1}} \frac12,$$ 
this yields 
\begin{multline*}
 \PP (\epsilon_1= 0 , \; A^{(1)}, \; \gamma_{S_{1}+1}=\pm1, \; \cap_{m=2}^n  A^{(m)}) )= \\ \frac1{3^{M_n}}  \PP ( \cap_{m=2}^n  A^{(m)},   \; \gamma_{S_{1}+1}=\pm1), 
\end{multline*}
 which can be written 
$$  \PP ( \epsilon_1= 0, \; \cap_{m=1}^n  A^{(m)} ) = \frac1{3^{M_1}} (u_2^n+v_2^n). $$ 
Similarly
\begin{align*}
 \PP (& \epsilon_1= 1 , \; A^{(1)}, \; \gamma_{S_{1}+1}=i_2, \; \cap_{m=2}^n  A^{(m)})
  \\
  &= \frac{ \PP ( \epsilon_1= 1,\; A^{(1)},\;  \gamma_{S_{1}+1}=i_2) \PP ( \cap_{m=2}^n  A^{(m)},   \; \gamma_{S_{1}+1}=i_2) }{\PP ( \gamma_{S_{1}+1}=i_2)} \\
  &=  \frac{M_1}{3^{M_1}} u_2^n.
\end{align*}
 Since $ \PP (\gamma_{S_{1}+1}=-i_2, \;  \epsilon_1= 1, \; A^{(1)}) =0,$ by summing the previous probabilities,  we obtain.  
$$ \PP (\cap_{m=1}^n  A^{(m)}) = \frac1{3^{M_1}} ((M_1+1) u_2^n+v_2^n)$$ and 
\begin{equation}
\label{eq:epsilon1}
\PP ( \epsilon_1= 1| \cap_{m=1}^n  A^{(m)})=\frac{M_1}{M_1+1 +x_2^n}. 
\end{equation}
Equation~(\ref{eq:epsilon1}) yields equation~(\ref{eq:eps1}) and  becomes
\begin{equation}
\label{eq:epsilon1_infinite}
\PP ( \epsilon_1= 1| a)=\frac{M_1}{M_1+1 +x_2^\infty}. 
\end{equation}
when $ K $ is infinite by letting $ n \to \infty.$
When $ N = 1, $ $  x_2^N $ is not defined but 
we may compute $ \PP ( \epsilon_1= 1|  A^{(1)} ) $ as follows. 
If $ N=1 $ it means that the vertex $a$ has a single block of constancy. 
Hence $  A^{(1)} = \{a\}, $ and the vertex $a$ has $ M_1 $ neighbors in 
the graph which are different from $ a$ and there are two edges that starts from $a$ 
and ends at $a.$ Hence, when $ N= 1, $ equation~(\ref{eq:epsilon1}) becomes
\begin{equation}
\label{eq:epsilon1_n=1}
\PP ( \epsilon_1= 1| A^{(1)} ) = \frac{M_1}{M_1+ 2},
\end{equation}
which is coherent with the convention $ x_2^1=1.$
 \section{Subsequent phases}
 \subsection{$\PP(\epsilon_n= 1|\epsilon_{n-1}= 0, \ldots, \epsilon_{n-(2k-1)}=1,\;  \cap_{m=1}^{N}  A^{(m)})$}
 Here we consider $ N > n \ge 2.$
 The numerator of the conditional probability is 
 \begin{align*}
  \PP(\epsilon_n &= 1, \; \epsilon_{n-1}= 0, \ldots, \epsilon_{n-(2k-1)}=1,\; \cap_{m=1}^{N}  A^{(m)}) \\ = \PP( &\epsilon_n= 1, \; \epsilon_{n-1}= 0, \ldots, \epsilon_{n-(2k-1)}=1, \; \cap_{m=1}^{N}  A^{(m)},\; \gamma_{S_n +1}= i_{n+1})  \\
 = 2 \PP(& \epsilon_n= 1 , \;\epsilon_{n-1}= 0, \ldots, \epsilon_{n-(2k-1)}=1,\; 
 \cap_{m=1}^n  A^{(m)}) \\ 
  &\phantom{2 \PP( \epsilon_n= 1 , \;\epsilon_{n-1}= 0, \ldots,}   \times \PP(\gamma_{S_n +1}= i_{n+1},\; \cap_{m=n+1}^{N}  A^{(m)})\\
 =  4 \PP(& \epsilon_{n-1}= 0, \ldots, \epsilon_{n-(2k-1)}=1,\; \cap_{m=1}^{n-1}  A^{(m)}, \; \gamma_{S_{n-1}+1}=i_n) \\
  &\phantom{2 \PP( \epsilon_n= 1 , \;\epsilon_{n-1}= 0, \ldots,}  \times \PP(\epsilon_n= 1, \; A^{(n)}) u_{n+1}^{N},  
\end{align*} 
where the last two equalities come from  conditional independence with respect to 
$\gamma$ taken at the convenient index. 
For the denominator of the conditional probability of the title of the section, the same kind of manipulations yield  
\begin{multline*}
\PP(\epsilon_{n-1}= 0, \ldots, \epsilon_{n-(2k-1)}=1, \; \cap_{m=1}^{N}  A^{(m)})= \\
  2 \PP(\epsilon_{n-1}= 0, \ldots, \epsilon_{n-(2k-1)}=1, \; \cap_{m=1}^{n-1}  A^{(m)})
 \PP(\gamma_{S_{n-1} +1}= i_{n},\; \cap_{m=n}^{N}  A^{(m)})  \\
  = 2 \PP(\epsilon_{n-1}= 0, \ldots, \epsilon_{n-(2k-1)}=1 , \; \cap_{m=1}^{n-1}  A^{(m)}) u_n^{N}.
\end{multline*}
Then
\begin{align*}
\PP(\epsilon_n= 1|\epsilon_{n-1}= 0, \ldots, \epsilon_{n-(2k-1)}=1,\;  \cap_{m=1}^{N}  A^{(m)}) &= 2 \PP(\epsilon_n= 1, \; A^{(n)}) \frac{u_{n+1}^{N}}{u_n^{N}},
\end{align*}
\begin{multline*}
\PP(\epsilon_n= 1|\epsilon_{n-1}= 0, \ldots, \epsilon_{n-(2k-1)}=1,\;  \cap_{m=1}^{N}  A^{(m)})= \\ \frac{M_n}{3^{M_n}}\frac{u_{n+1}^{N}}{\frac1{3^{M_n}}(M_n u_{n+1}^{N} + v_{n+1}^{N})} ,
\end{multline*}
\begin{align}
\PP(\epsilon_n= 1|\epsilon_{n-1}= 0, \ldots, \epsilon_{n-(2k-1)}=1,\;  \cap_{m=1}^{N}  A^{(m)}) &=M_n \frac1{M_n + x_{n+1}^{N}} \notag\\
&=M_n x_n^{N}. \label{eq:ep=1-ep=1}
\end{align}
When $K$ is infinite, we can let $ N \to \infty$ to obtain  
\begin{equation}
\label{eq=0_infty}
\PP(\epsilon_n= 1|\epsilon_{n-1}= 0, \ldots, \epsilon_{n-(2k-1)}=1,\;  \cap_{m=1}^{\infty}  A^{(m)}) = M_n x_n^{\infty}. 
\end{equation}

When $ N = n,$ we have to change the computation of the numerator. Actually in this case
\begin{align*}
 \PP(\epsilon_n = 1&, \; \epsilon_{n-1}= 0, \ldots, \epsilon_{n-(2k-1)}=1, \; \cap_{m=1}^{n}  A^{(m)})= \\2 &\PP(\epsilon_{n-1}= 0, \ldots, \epsilon_{n-(2k-1)}=1,\;  \cap_{m=1}^{n-1}  A^{(m)}) \PP(\epsilon_n= 1, \; A^{(n)}). \\
\end{align*} 
Since 
\begin{align*}
\PP(\epsilon_{n-1}= 0&, \ldots, \epsilon_{n-(2k-1)}=1,\;  \cap_{m=1}^{n}  A^{(m)})= \\
  =2 \PP(\epsilon_{n-1}&= 0, \ldots, \epsilon_{n-(2k-1)}=1 \cap_{m=1}^{n-1}  A^{(m)})
 \PP(\gamma_{S_{n-1} +1}= i_{n},\; \cap  A^{(n)})  \\
  =2 \PP(\epsilon_{n-1}&= 0, \ldots, \epsilon_{n-(2k-1)}=1, \; \cap_{m=1}^{n-1}  A^{(m)}) u_n^n, 
\end{align*}
we get 
\begin{align*}
\PP(\epsilon_n= 1|\epsilon_{n-1}= 0, \ldots, \epsilon_{n-(2k-1)}=1,\;  \cap_{m=1}^{n}  A^{(m)}) &= 2 \frac{\PP(\epsilon_n= 1, \; A^{(n)})}{u_n^n} \\
&=  \frac{M_n}{3^{M_n}}\frac1{u_n^n}.
\end{align*}
Because of~(\ref{u_n^n})
\begin{equation}
\label{eq=0_der}
\PP(\epsilon_n= 1|\epsilon_{n-1}= 0, \ldots, \epsilon_{n-(2k-1)}=1,\;  \cap_{m=1}^{n}  A^{(m)}) = \frac{M_n}{M_n+1}.
\end{equation}
Hence equation~(\ref{eq:ep=1-ep=1}) is always verified. 
\subsection{$\PP(\epsilon_n= 1|\epsilon_{n-1}= 0, \ldots, \epsilon_1=0,\;  \cap_{m=1}^{N}  A^{(m)})$}
Let us start with the denominator and assume until further notice that $ N > n \ge 2,$ 
\begin{multline*}
 \PP(\epsilon_{n-1}= 0, \ldots, \epsilon_1=0,\;  \cap_{m=1}^{N}  A^{(m)}) \\ = 
\sum_{\nu= -1}^{+1} \PP(\gamma_1=\nu,\; \epsilon_{n-1}= 0, \ldots, \epsilon_1=0,\;  \cap_{m=1}^{N}  A^{(m)}).
\end{multline*}
Then 
\begin{multline*}
 \PP(\gamma_1=\nu,\; \epsilon_{n-1}= 0, \ldots, \epsilon_1=0,\;  \cap_{m=1}^{N}  A^{(m)}) =\\ 2  \PP(\gamma_1=\nu,\; \epsilon_1=0,\; A^{(1)}) \PP(\epsilon_{n-1}= 0, \ldots, \epsilon_2=0,\; \gamma_{S_1+1}=\nu,\; \cap_{m=2}^{N}  A^{(m)}) \\
  =   \PP(\epsilon_1=0 ,\; A^{(1)}) \PP(\epsilon_{n-1}= 0, \ldots, \epsilon_2=0,\; \gamma_{S_1+1}=\nu,\; \cap_{m=2}^{N}  A^{(m)}) \\
  = \prod_{m=1}^{n-1} \PP(\epsilon_m=0 ,\; A^{(m)}) \PP (\gamma_{S_{n-1}+1}=\nu, \;  \cap_{m=n}^{N}  A^{(m)}).
\end{multline*}
Hence 
\begin{equation}
\label{eq=ep=0}
 \PP(\epsilon_{n-1}= 0, \ldots, \epsilon_1=0,\;  \cap_{m=1}^{N}  A^{(m)})  = 
 \prod_{m=1}^{n-1} \PP(\epsilon_m=0 ,\; A^{(m)}) (u_n^N+v_n^N).
\end{equation}
The previous formula is also true for $ N =n.$
Moreover
\begin{multline*}
\PP(\epsilon_n= 1,\; \epsilon_{n-1}= 0, \ldots, \epsilon_1=0, \cap_{m=1}^{N}  A^{(m)})
 \\=  \PP(\epsilon_1= 0,\; \; \gamma_1 = i_n, \; \epsilon_2= 0,\;   \; \gamma_{S_1+1} = i_n, \; \ldots,  \epsilon_{n-1}= 0,\dots\\ \phantom{\PP(\epsilon_1= 0,\; \; \gamma_1 = i_n)}\ldots, \gamma_{S_{n-1}+1} = i_n \;  \epsilon_{n}= 1, \; \gamma_{S_{n}+1} = i_{n+1} \;  \cap_{m=1}^{N}  A^{(m)}) \\
= \prod_{m=1}^{n-1} \PP(\epsilon_m=0 ,\; A^{(m)}) \PP(\epsilon_{n}= 1, \; \gamma_{S_{n}+1} = i_{n+1} \;  \cap_{m=n}^{N}  A^{(m)}))\\
= 2 \prod_{m=1}^{n-1} \PP(\epsilon_m=0 ,\; A^{(m)}) \PP(\epsilon_{n}= 1, \; A^{(n)}) u_{n+1}^{N},
\end{multline*} 
for  $ N > n.$ Hence 
\begin{align}
\PP(\epsilon_n= 1|\epsilon_{n-1}= 0, \ldots, \epsilon_1=0,\;  \cap_{m=1}^{N}  A^{(m)}) &= \PP(\epsilon_{n}= 1, \; A^{(n)}) \frac{2 u_{n+1}^{N}}{u_{n}^{N}+ v_{n}^{N}} \notag \\
&= \frac12 \frac{M_n}{3^{M_n}} \frac{ 2 \times 3^{M_n}  u_{n+1}^{N} }{(M_n +1) u_{n+1}^{N} +  v_{n+1}^{N}} \notag \\
&= \frac{M_n}{M_n + 1 + x_{n+1}^{N} } \label{eq:ep=1-ep=0}
\end{align}
Where we have used 
$$ \PP(\epsilon_{n}= 1, \; A^{(n)}) = \frac12 \frac{M_n}{3^{M_n}}  $$ 
 and Lemma~\ref{lem:rec-matrix}  in the previous computations. 
When $ K = \infty$ we let $  N \to \infty$  in~(\ref{eq:ep=1-ep=0})
and we get 
 \begin{equation}
 \label{eq=ep=0, infty}
 \PP(\epsilon_n= 1|\epsilon_{n-1}= 0, \ldots, \epsilon_1=0,\;  \cap_{m=1}^{\infty}  A^{(m)}) = \frac{M_n}{M_n + 1 + x_{n+1}^{\infty} }.
\end{equation}
When $ N=n,$
\begin{multline*}
 \PP(\epsilon_n= 1,\; \epsilon_{n-1}= 0, \ldots, \epsilon_1=0, \cap_{m=1}^{N}  A^{(m)}) \\ =  \prod_{m=1}^{n-1} \PP(\epsilon_m=0 ,\; A^{(m)}) \PP(\epsilon_{n}= 1, \; A^{(n)}).
\end{multline*}
 \begin{align}
 \PP(\epsilon_n= 1|\epsilon_{n-1}= 0, \ldots, \epsilon_1=0,\;  \cap_{m=1}^{n}  A^{(m)}) &= \frac{\PP(\epsilon_{n}= 1, \; A^{(n)})}{u_n^n + v_n^n}\\
 &= \frac12 \frac{M_n}{3^{M_n}}  \frac1{u_n^n + v_n^n} \\
 &= \frac{M_n}{M_n + 2}. \label{eq=ep=0,l=0}
 \end{align}
Equations~(\ref{eq:ep=1-ep=0}) and~(\ref{eq=ep=0,l=0}) yield~(\ref{eq:sub-1}) in view of  the convention $x_{n+1}^n=1.$ 
 The equation~(\ref{eq:loop}) is a consequence of the definition of  $\epsilon_m.$
 
 To conclude the proof let us prove that $ \forall m \ge 1 $ the distribution of the $ E_m$'s is uniform on $ \{ 1,\ldots,M_m\} $  conditionally to the event $\epsilon_m = 1.$
 Let us fix $ m_0 $ such that $ 2 \le m_0 \le N.$
 We have to show that $ \forall j_0 \in \{ 1,\ldots,M_{m_0} \} $ 
 $$ \PP (E_{m_0}=j_0 | a, \; \epsilon_{m_0}) $$ 
 does not depend on $ j_0.$ 
 We have
 \begin{multline*}
\PP(E_{m_0}=j_0 , \; \cap_{m=1}^N A^{(m)}, \;  \epsilon_{m_0}=1)= \\ \PP ( E_{m_0}=j_0 
, \;  \epsilon_{m_0}=1, \; \gamma_{S_{m_0-1}+1}= i_{m_0}, 
\; \gamma_{S_{m_0}+1}= i_{m_0 + 1},\; \cap_{m=1}^N A^{(m)}) 
\end{multline*}
\begin{multline*}
= 2 \PP ( \cap_{m=1}^{m_0 -1} A^{(m)}, \; \gamma_{S_{m_0-1}+1}= i_{m_0})  \\
\times  \PP (  E_{m_0}=j_0 ,   \epsilon_{m_0}=1,   \gamma_{S_{m_0-1}+1}= i_{m_0},  \gamma_{S_{m_0}+1}= i_{m_0 + 1}, \cap_{m=m_0 +1}^{N} A^{(m)}) 
\end{multline*}
\begin{multline*}
= 4 \PP ( \cap_{m=1}^{m_0 -1} A^{(m)}, \; \gamma_{S_{m_0 -1}+1}= i_{m_0}) \\
\times \PP (  E_{m_0}=j_0 , \;  \epsilon_{m_0}=1, \;A^{(m_0)})  \PP ( \gamma_{S_{m_0-1}+1}= i_{m_0 + 1}, \;  \cap_{m=m_0 +1}^{N} A^{(m)}) 
\end{multline*}
thanks to the conditional independence when $\gamma$ is given. 
Furthermore $$ \PP(E_{m_0}=j_0 , \; A^{(m_0)}, \;  \epsilon_{m_0}=1) = \PP (\alpha_{j_0}=1,\; \forall i \neq j_0,\; \alpha_i = 0, \; A^{(m_0)}) $$
which does not depend on $ j_0 $ and consequently  $ \PP(E_{m_0}=j_0 , \; A^{(m_0)}, \;  \epsilon_{m_0}=1) $ does not depend on $ j_0.$
It is also the case for $\PP (E_{m_0}=j_0 , \; \cap_{m=1}^N A^{(m)}, \;  \epsilon_{m_0}=1 | a, \; \epsilon_{m_0}=1) $ since $ \PP ( \cap_{m=1}^N A^{(m)}, \;  \epsilon_{m_0}=1) $ does not depend on $ j_0.$ The proof is easier when $ m_0=1.$

\subsection*{Acknowledgment}
Th authors would like to thank James Norris for fruitful discussions concerning a previous version of the article. 




\bibliographystyle{plain}
\bibliography{biblio6,ref3}

\end{document}